\documentclass[11pt]{article}
\usepackage{amsmath,amssymb,amsthm}
\usepackage{color}
\usepackage{pdfsync}
\usepackage[english]{babel}
\usepackage[utf8]{inputenc}

\usepackage{fancybox}
\usepackage{listings}
\usepackage{mathtools}
\usepackage[bottom]{footmisc}
\usepackage{verbatim}

\usepackage{float}
\usepackage{makeidx}

\normalbaselines

\newtheorem{Theorem}{Theorem}[section]

\newtheorem{Definition}[Theorem]{Definition}

\newtheorem{Proposition}[Theorem]{Proposition}

\newtheorem{Lemma}[Theorem]{Lemma}

\newtheorem{Corollary}[Theorem]{Corollary}

\newtheorem{Remark}[Theorem]{Remark}

\newcommand{\N}{\mathbb N}

\newcommand{\RR}{{{\rm I} \kern -.15em {\rm R} }}
\newcommand{\R}{{{\rm I} \kern -.15em {\rm R} }}

\newcommand{\C}{{{\rm l} \kern -.42em {\rm C} }}

\newcommand{\nat}{{{\rm I} \kern -.15em {\rm N} }}

\newcommand{\be}{\begin{equation}}
\newcommand{\ee}{\end{equation}}
\newcommand{\beq}{\begin{eqnarray}}
\newcommand{\eeq}{\end{eqnarray}}
\newcommand{\beqs}{\begin{eqnarray*}}
\newcommand{\eeqs}{\end{eqnarray*}}
\newcommand{\bt}{\begin{Theorem}}
\newcommand{\et}{\end{Theorem}}
\newcommand{\br}{\begin{Remark}}
\newcommand{\er}{\end{Remark}}
\newcommand{\bc}{\begin{Corollary}}
\newcommand{\ec}{\end{Corollary}}
\newcommand{\bl}{\begin{Lemma}}
\newcommand{\el}{\end{Lemma}}
\newcommand{\bd}{\begin{definition}}
\newcommand{\ed}{\end{definition}}
\newcommand{\bp}{\begin{Proposition}}
\newcommand{\eP}{\end{Proposition}}

\title{Opinion dynamics under common influencer assumption\\ or leadership control}

\bigskip

\author{Chiara Cicolani\footnote{Dipartimento di Ingegneria e Scienze dell'Informazione e Matematica\\ Universit\`a dell''Aquila\\
Via Vetoio,  67100 L'Aquila - Italy\\ Email: chiara.cicolani@graduate.univaq.it},
\hspace{0.2cm}
Badis Ouahab \footnote{ \'Ecole  Militaire Polytechnique, 16046 Algiers, Algeria \\   						            Email: badisouahab77@gmail.com}
\&\hspace{0.2cm}Cristina Pignotti\footnote{Dipartimento di Ingegneria e Scienze dell'Informazione e Matematica\\ Universit\`a dell''Aquila\\
Via Vetoio,  67100 L'Aquila - Italy\\ Email: cristina.pignotti@univaq.it}}
	
\begin{document}
\maketitle

\begin{abstract}
We study Hegselmann-Krause type opinion formation models with non-universal interaction and time-delayed coupling. We assume the presence of a common influencer between two different agents. Moreover, we explore two cases in which such an assumption does not hold but leaders with independent opinion are present. By using careful estimates on the system's trajectories, we are able to prove asymptotic convergence to consensus estimates. Some numerical tests illustrate the theoretical results.
\end{abstract}

\section{Introduction}
In recent years, multiagent systems have become a highly attractive research topic due to their numerous applications in various scientific fields. They naturally appear in biology \cite{Cama, Carrillo, CS1}, ecology \cite{Sole}, economics \cite{Carrillo, Marsan}, social sciences \cite{Bellomo, BN, BHT, Castellano, Campi}, physics \cite{DH, Stro}, control theory and optimization \cite{Basco,D, PaolucciP, PRT, WCB}, engineering, and robotics \cite{Bullo, Desai}. For additional applications, see also \cite{Hel, Jack, XWX}.

A significant feature often analyzed is the potential emergence of self-organization, leading the group's agents to exhibit globally collective behaviors. Here, we focus on the renowned Hegselmann-Krause model for opinion formation, originally proposed in \cite{HK}. Since then, several generalizations have been proposed (see, for instance, \cite{Bellomo, BN, CFT, CFT2, Ceragioli, JM}).

One of the most natural extensions of the Hegselmann-Krause model involves the analysis of time-delayed interactions, to account for the time required for each agent to receive information from other agents or reaction times. Opinion formation models in the presence of time delay effects have already been studied by several authors, see, for example, \cite{CPP,CicPig,CPM2A,onoff,3,4,P,PaolucciP}. 
 Regarding the second-order version of the Hegselmann-Krause model, namely the Cucker-Smale model \cite{CS1}, introduced to describe flocking phenomena, delayed interactions have also been considered in various works, see, for example, 
\cite{Chen,CH,CL,CP,CCP,Cont,EHS,HM,PR,PT,Cartabia}.

In this paper, we focus on studying the convergence to consensus for some Hegselmann-Krause opinion formation models with time-delayed coupling. We assume that the time delay between two agents is dependent on the agents themselves. Specifically, we examine scenarios where interactions are not universal but where a common influencer always exists between any two different agents.

Consider a finite set of $N \in \N, N \geq 3,$ agents. Let $x_i(t) \in \R^d$ be the opinion of the i-th agent at time $t.$ We assume that there is a time lag in the interaction between the agents, described by  positive constants $\tau_{ij}, \ i,j=1,\dots,N.$ Moreover, we assume that the interaction is non-universal. Then, the interaction between the particles of the system is described by the following Hegselmann-Krause type model:
\begin{equation}\label{eq1}
\frac{d}{d t}x_i(t)= \frac{1}{N-1}\sum_{j \neq i } \chi_{ij} a_{ij}(t)(x_j(t-\tau_{ij})-x_i(t)), \quad t>0, \ i=1,...,N, 
\end{equation}
with weights $a_{ij}(t)$ of the form 
\begin{equation} \label{a}
 a_{ij}(t):=\psi_{ij}(x_i(t),x_j(t-\tau_{ij})), \quad \  i,j=1,...,N,
\end{equation}
where the influence functions $\psi_{ij}:\R^d\times \R^d \rightarrow \R$ are positive, bounded and continuous. We denote
$$
K_{ij}:= \parallel \psi_{ij} \parallel_{\infty}, \quad \forall i,j=1,\dots,N,
$$
and let 
\begin{equation} \label{K}
K:=\max_{i,j=1,\dots,N} K_{ij}.
\end{equation}
The function $\chi_{ij}$ for $i,j=1,\dots,N$ is defined as
\begin{equation} \chi_{ij}=\begin{cases} 1 \quad \mbox{if $j$ transmits information to $i$,} \\
0 \quad  \mbox{otherwise}. \end{cases} 
\end{equation}
Moreover, let us assume the following initial conditions: 
\begin{equation} \label{IC}
x_i(t)=x_i^0(t), \quad i=1,\dots,N, \ t \in [-\tau,0],
\end{equation}
where 
$$\tau:=\max_{i,j=1,\dots,N} \tau_{ij},$$
and  $x_i^0:[-\tau, 0]\rightarrow \R^d, \  i=1, \dots, N,$ are continuous functions. For existence results about the above model we refer to the classical books \cite{Halanay,Hale}. Here, we are interested in convergence to consensus results.

We make the following assumption over the structure of the system:

\textbf{(CI)} For all $i,j \in \left\{1,\dots,N \right\}$ there exists $k \in \left\{1,\dots,N \right\}$ such that $\chi_{ik}=\chi_{jk}=1$ and $\tau_{ik}=\tau_{jk}.$  This means that for each pair of agents $(x_i, x_j)$ there's another agent $x_k$, that we call {\em common influencer of  $x_i$ and $x_j,$ } that transmits information to both them with the same time delay.

 Let us define the \emph{diameter function} as
$$d(t):= \max_{i,j=1,\dots,N} |x_i(t)-x_j(t)|.$$
\begin{Definition}\label{def1} We say that a solution converges to consensus if $$ \lim_{t \rightarrow \infty} d(t)=0$$
\end{Definition}

After preliminary estimates, we will prove the convergence to consensus for solutions to system \eqref{eq1}. As an example falling in the previous setting, we will consider a model for a population with a leader subset.
When the leaders' numbers is $1$ or $2,$ the common influencer assumption is not satisfied. So, we will study these cases using different appropriate arguments. In the case of a unique leader, we will also analyze a control problem (cf.  \cite{PaolucciP,WCB}).

The rest of the paper is organized as follows.  In Section \ref{Prel}, we introduce some notations and provide preliminary lemmas. In Section \ref{proof}, we state and prove the consensus result under the assumption {\bf (CI)}. Sections \ref{oneleader} and \ref{twoleaders} explore different systems for which the Common Influencer Assumption does not hold. In Section \ref{oneleader}, we consider the presence of a leader that influences all other agents but is not influenced by anyone else, examining both the case of a constant trajectory for the leader and a controlled trajectory. Section \ref{twoleaders} deals with the scenario involving two leaders. Finally, in Section \ref{num}, we present numerical tests that validate the theoretical results.

\section{Preliminaries} \label{Prel}
 In this section, we present some preliminary results concerning system \eqref{eq1}. We omit the  proofs since they are analogous to the ones of  Lemma 2.2, Lemma 2.3, and Lemma 2.5 of \cite{CPM2A}.

\begin{Lemma} \label{2.1}
Let $\{x_i(t)\}_{i=1}^N$ be a solution to \eqref{eq1}-\eqref{IC}. Then, for each $v \in \R^d$ and $T \geq 0 $, we have 
\begin{equation} \label{eq4}
\min_{j=1,...,N} \min_{s \in [T-\tau,T]} \langle x_j(s), v \rangle \leq \langle x_i(t), v \rangle \leq \max_{j=1,...,N} \max_{s \in [T-\tau,T]}\langle x_j(s), v \rangle, 
\end{equation}
for all $t\ge T-\tau,$ $i=1, \dots, N.$
\end{Lemma}

We now introduce the following notation. 
\begin{Definition} \label{def}
We define 
\begin{equation}\label{D0}
D_0:=\max_{i,j=1,\dots,N} \max_{s,t \in [-\tau,0]} \vert x_i(s)-x_j(t) \vert,
\end{equation}
and in general for all $n \in \N,$
\begin{equation}\label{Dn}
D_n:=\max_{i,j=1,\dots,N} \max_{s,t \in [n\tau-\tau,n\tau]} \vert x_i(s)-x_j(t) \vert.
\end{equation}
\end{Definition}

Let us denote with $\N_0:= \N \cup \{0\}.$
\begin{Lemma} \label{2.2}
Let $\{x_i(t)\}_{i=1}^N$ be a solution to \eqref{eq1}-\eqref{IC}.
For each $n \in \N_0$ and $\forall \ i,j=1,...,N$ we get
\begin{equation}\label{eq12}
|x_i(t)-x_j(t)| \leq D_n, \quad \forall \ t \geq n\tau-\tau.\
\end{equation}
\end{Lemma}

\begin{Remark} \label{2.3}
Let us note that from Lemma \ref{2.2} it follows, in particular, 
$$ \vert x_i(s)-x_j(t) \vert \leq D_0, \quad \forall s,t \geq -\tau, \ \forall i,j=1,\dots,N.$$
Moreover, it holds that 
\begin{equation} \label{decr_prop}
D_{n+1} \leq D_n, \quad \forall n \in \N_0.
\end{equation}
\end{Remark}
As in \cite{CPM2A}, 
one can find a bound on $| x_i(t)|$ uniform with respect to $t$ and $i=1,...,N$. We have the following lemma (cf. \cite[Lemma 2.5]{CPM2A}).

\begin{Lemma} \label{2.4}
Let $\{x_i(t)\}_{i=1}^N$ be a solution to \eqref{eq1}-\eqref{IC}.
For all $i=1,...,N$ we have
\begin{equation}
|x_i(t)| \leq M_0, \quad \forall \ t \geq 0,
\label{eq16}
\end{equation}
where
$$ M_0:=\max_{i=1,...,N} \max_{s \in [-\tau,0]} |x_i(0)|.$$
\end{Lemma}

\begin{Remark} \label{2.5}
    From Lemma \ref{2.3}, since the influence functions $\psi_{ij}$ are continuous, we deduce that
\begin{equation} \label{lower_bound}
\psi_{ij}(x_i(t),x_j(t-\tau_{ij})) \geq \psi_0:= \min_{i,j=1, \dots, N}\min_{|z_1|,|z_2| \leq M_0} \psi_{ij}(z_1,z_2) >0  , 
\end{equation}
for each $t \geq 0 $ and $i,j=1,...,N$.
\end{Remark}
Now, we need the following estimates.
\begin{Lemma} \label{2.6}
Let $\{x_i(t)\}_{i=1}^N$ be a solution to \eqref{eq1}-\eqref{IC}. 
For all $i,j=1,\dots,N,$ unit vector $ v \in \R^d$ and $n \in \N_0$ we have that 
\begin{equation}
\langle x_i(t)-x_j(t),v \rangle \leq e^{-K(t-t_0)} \langle x_i(t_0)-x_j(t_0),v \rangle + (1-e^{-K(t-t_0)})D_n,
\label{eq18}
\end{equation}
for all $t \geq t_0 \geq n\tau.$ Moreover, for all $n \in \N_0$ we get
\begin{equation}
D_{n+1} \leq e^{-K \tau} d(n\tau) + (1-e^{-K\tau})D_n. 
\label{eq20}
\end{equation}
\end{Lemma}
\begin{proof}
Let us fix $v \in \R^d$ such that $|v|=1$ and $n \in \N_0.$ We denote by
$$ M_n := \max_{i=1,...,N} \max_{t \in [n\tau-\tau,n\tau]} \langle x_i(t), v \rangle,$$
and 
$$ m_n:= \min_{i=1,...,N} \min_{t \in [n\tau-\tau, n\tau]} \langle x_i(t), v \rangle.$$
Note that $M_n-m_n \leq D_n.$ Now, fix $i=1,\dots,N$ and $t \geq t_0 \geq n\tau.$ Then, by definition of the system \eqref{eq1} and applying Lemma \ref{2.1} with $T=n\tau,$ we have that
\begin{equation*}
\begin{split}
\frac{d}{d t} \langle x_i(t),v \rangle & = \frac{1}{N-1} \sum_{j \neq i} \chi_{ij} a_{ij}(t) \langle x_j(t-\tau_{ij})-x_i(t),v \rangle \\
& \leq \frac{1}{N-1}\sum_{j \neq i} a_{ij}(t) (M_n - \langle x_i(t),v \rangle) \\
& \leq K(M_n - \langle x_i(t),v \rangle) \\
\end{split}
\end{equation*}
where we used that $t-\tau_{ij}\geq n\tau-\tau$ and that
$$ \sum_{j \neq i} \chi_{ij} \leq N-1.$$
By applying the Grönwall's Lemma, we find that 
\begin{equation}\label{eq21}
\langle x_i(t),v \rangle \leq e^{-K(t-t_0)}  \langle x_i(t_0),v \rangle + (1-e^{-K(t-t_0)})M_n .
\end{equation}
On the other hand, for $i=1,\dots,N$ and $t \geq t_0 \geq n\tau,$ following a similar argument, we have that 
\begin{equation*}
\frac{d}{d t} \langle x_j(t),v \rangle \geq K(m_n - \langle x_j(t),v \rangle)
\end{equation*}
 and, applying again the Grönwall's Lemma, 
\begin{equation}\label{eq22}
\langle x_j(t),v \rangle \geq e^{-K(t-t_0)}  \langle x_j(t_0),v \rangle + (1-e^{-K(t-t_0)})m_n .
\end{equation}
From (\ref{eq21}) and (\ref{eq22}), we have that
\begin{equation*}
\begin{split}
\langle x_i(t)-x_j(t),v \rangle & \leq e^{-K(t-t_0)}  \langle x_i(t_0)-x_j(t_0),v \rangle + (1-e^{-K(t-t_0)})(M_n-m_n) \\
& \leq e^{-K(t-t_0)}\langle x_i(t_0)-x_j(t_0),v \rangle + (1-e^{-K(t-t_0)})D_n.
\end{split}
\end{equation*}
So, we have (\ref{eq18}). The inequality \eqref{eq20} follows as in the second part of the proof of Lemma 2.6 in \cite{CPM2A}.
\end{proof}
The following result is crucial in order to prove the exponential convergence to consensus estimate.
    
\begin{Lemma} \label{2.7}
Let $\{x_i(t)\}_{i=1}^N$ be a solution to \eqref{eq1}-\eqref{IC} and let us assume {\bf (CI)}.
Then, there exists a constant $C \in (0,1)$ such that 
\begin{equation} \label{estimate_diam}
d(n\tau) \leq C D_{n-2},
\end{equation}
for all $n \geq 2$.
\end{Lemma}
\begin{proof}
Notice that if $d(n\tau)=0$, the result is trivial. Suppose $d(n\tau)>0$ and let be $i,j=1,\dots,N$ such that $d(n\tau)= \vert x_i(n\tau)-x_j(n\tau)\vert.$ We set 
    $$ v:=\frac{x_i(n\tau)-x_j(n\tau)}{|x_i(n\tau)-x_j(n\tau)|}.$$
Then, $d(n\tau)=\langle x_i(n\tau)-x_j(n\tau), v \rangle$. Consider,
    $$ M_{n-1} := \max_{i=1,...,N} \max_{s \in [(n-2)\tau,(n-1)\tau]} \langle x_i(s), v \rangle $$
            and 
    $$ m_{n-1}:= \min_{i=1,...,N} \min_{s \in [(n-2)\tau,(n-1)\tau]} \langle x_i(s), v \rangle. $$
It holds that $M_{n-1}-m_{n-1} \leq D_{n-1}$. \\
From \eqref{eq1}, for $t \in [(n-1)\tau,n\tau],$ we can write
\begin{equation}
\begin{split}
& \frac{d}{d t} \langle x_i(t)-x_j(t), v \rangle  = \frac{1}{N-1} \sum_{l:l\neq i} \chi_{il}a_{il}(t) \langle x_l(t-\tau_{il})-x_i(t), v \rangle \\
& \hspace{3.5 cm} - \frac{1}{N-1}\sum_{l:l\neq j} \chi_{jl}a_{jl}(t) \langle x_l(t-\tau_{jl})-x_j(t), v \rangle \\
& = \frac{1}{N-1}\sum_{l:l\neq i} \chi_{il} a_{il}(t) (\langle x_l(t-\tau_{il}), v \rangle - M_{n-1}) \\
& \hspace{3.5 cm} + \frac{1}{N-1}\sum_{l:l\neq i} \chi_{il}a_{il}(t) (M_{n-1}-\langle x_i(t), v \rangle) \\
& - \frac{1}{N-1} \sum_{l:l\neq j} \chi_{jl} a_{jl}(t) (\langle x_l(t-\tau_{jl}), v \rangle - m_{n-1}) \\
& \hspace{3.5 cm}- \frac{1}{N-1}\sum_{l:l\neq j} \chi_{jl}a_{jl}(t) (m_{n-1}-\langle x_j(t), v \rangle) .
\end{split}
\label{eq26} 
\end{equation}
Let us define the sums $S_1$ and $S_2$ and, using Remark \ref{2.5}, we have 
\begin{equation}
\begin{split}
S_1 & := \frac{1}{N-1}\sum_{l:l\neq i}\chi_{il} a_{il}(t) (\langle x_l(t-\tau_{il}), v \rangle - M_{n-1}) \\
& + \frac{1}{N-1} \sum_{l:l\neq i}\chi_{il} a_{il}(t) (M_{n-1}-\langle x_i(t), v \rangle) \\
& \leq \frac{\psi_0}{N-1}  \sum_{l:l\neq i} \chi_{il}(\langle x_l(t-\tau_{il}), v \rangle - M_{n-1}) + K (M_{n-1}-\langle x_i(t), v \rangle),
\end{split}
\label{eq27}
\end{equation}
and
\begin{equation}
\begin{split}
S_2 & := \frac{1}{N-1}\sum_{l:l\neq j} \chi_{jl} a_{jl}(t) (m_{n-1}-\langle x_l(t-\tau_{jl}), v \rangle) \\
& + \frac{1}{N-1} \sum_{l:l\neq j} \chi_{jl} a_{jl}(t) (\langle x_j(t), v \rangle - m_{n-1}) \\
& \leq \frac{\psi_0}{N-1} \sum_{l:l\neq j} \chi_{jl}(m_{n-1}-\langle x_l(t-\tau_{jl}), v \rangle)+ K(\langle x_j(t), v \rangle - m_{n-1}),\\
\end{split}
\label{eq28}
\end{equation}
where we used the fact that, being $t \in [(n-1)\tau,n\tau],$ it holds that  \\ $t -\tau_{ij} \in [(n-2)\tau,n\tau], \ \forall ij=1,\dots,N,$ and then 
\begin{equation*}
m_{n-1} \leq \langle x_l(t-\tau_{il}),v \rangle \leq M_{n-1}, \quad m_{n-1} \leq \langle x_l(t-\tau_{jl}),v \rangle \leq M_{n-1}, 
\end{equation*} 
 for all $l=1,\dots,N.$
Then, using \eqref{eq27} and \eqref{eq28} in \eqref{eq26}, we have that 
\begin{equation*}
\begin{split}
& \frac{d}{d t} \langle x_i(t)-x_j(t), v \rangle  \leq K (M_{n-1}-m_{n-1}) - K \langle x_i(t)-x_j(t), v \rangle \\
& + \frac{\psi_0}{N-1}\sum_{l:l\neq i} \chi_{il}(\langle x_l(t-\tau_{il}), v \rangle - M_{n-1}) \\
& + \frac{\psi_0}{N-1} \sum_{l:l \neq j} \chi_{jl}(m_{n-1} - \langle x_l(t-\tau_{jl}), v \rangle). \\
\end{split}
\end{equation*}
Now, we use the assumption \textbf{(CI)}. Let $x_k,$ for some $k \in \{1,\dots,N\},$ be a  common influencer between $x_i$ and $x_j.$ Using that, for all $l=1, \dots, N, $
$$ \langle x_l(t-\tau_{il}), v \rangle - M_{n-1} \le 0$$
and 
$$  m_{n-1} - \langle x_l(t-\tau_{jl}), v \rangle\le 0,$$
we find that
\begin{equation*}
\begin{split}
& \frac{d}{d t} \langle x_i(t)-x_j(t), v \rangle  \leq K(M_{n-1}-m_{n-1}) - K \langle x_i(t)-x_j(t), v \rangle \\
& +\frac{\psi_0}{N-1} (\langle x_k(t-\tau_{ik}),v \rangle-M_{n-1}+m_{n-1}-\langle x_k(t-\tau_{jk}), v \rangle) \\
& =\left( K - \frac{\psi_0}{N-1}\right)(M_{n-1}-m_{n-1}) - K \langle x_i(t)-x_j(t), v \rangle .
\end{split}  
\end{equation*}
Applying the Grönwall's Lemma in $[(n-1)\tau,t],$ with $t \in [(n-1)\tau,n\tau],$ we find that
\begin{equation*}
\begin{split}
\langle x_i(t)-x_j(t), v \rangle & \leq e^{-K(t-n\tau+\tau)} \langle x_i(n\tau-\tau)-x_j(n\tau-\tau), v \rangle \\
& + \left( 1- \frac{\psi_0}{K(N-1)}\right)(M_{n-1}-m_{n-1}) (1-e^{-K(t-n\tau+\tau)}).
\end{split}
\end{equation*}
Since this is valid $\forall \ t \in [(n-1)\tau,n\tau], $ let consider $t=n\tau.$ Then,
\begin{equation} \label{2.37}
\begin{split} 
d(n&\tau)  \leq e^{-K\tau} \langle x_i(n\tau-\tau)-x_j(n\tau-\tau), v \rangle \\
& \quad\quad + \left( 1- \frac{\psi_0}{K(N-1)}\right)(M_{n-1}-m_{n-1}) (1-e^{-K\tau}) \\
& \leq e^{-K\tau} |x_i(n\tau-\tau)-x_j(n\tau-\tau)| |v| \\
& + \left( 1- \frac{\psi_0}{K(N-1)}\right)(M_{n-1}-m_{n-1}) (1-e^{-K\tau}) \\
& \leq D_{n-1} \left [e^{-K\tau}+1- \frac{\psi_0}{K(N-1)}(1-e^{-K\tau}) \right ] \\
& \leq  D_{n-2} \left [1- \frac{\psi_0}{K(N-1)}(1-e^{-K\tau}) \right ].
\end{split}
\end{equation}
Therefore, \eqref{estimate_diam} follows with
\begin{equation} \label{C}
C:= 1-\frac{\psi_0}{K(N-1)}(1-e^{-K\tau}).
\end{equation} 
\end{proof}
\section{Consensus under Common Influencer assumption} \label{proof}
We are ready to give our convergence to consensus estimate.
\begin{Theorem}\label{main}
    Assume {\bf (CI)}. Then, every solution
$\{x_i(t)\}_{i=1}^N$ to \eqref{eq1}-\eqref{IC} 
 satisfies the exponential decay estimate
    \begin{equation} \label{exp_decay}
        d(t)\leq {D_0}e^{-\gamma (t-2\tau)} \ \mbox{for all} \ t \geq 0,
    \end{equation}
    for a suitable positive constant $\gamma.$
\end{Theorem}
\begin{proof}
Let $\left \{x_i(t)\right\}_{i=1}^N$ be the solution of \eqref{eq1}-\eqref{IC}. We claim that 
\begin{equation} \label{claim}
D_{n+1}\leq \tilde{C} D_{n-2}, \ \forall n \geq 2,
\end{equation}
for some constant $\tilde{C} \in (0,1).$ Using \eqref{decr_prop},\eqref{eq20} and \eqref{estimate_diam} we get 
\begin{equation*}
\begin{split}
D_{n+1} & \leq e^{-K\tau}d(n\tau)+(1-e^{-K\tau})D_n \\
& \leq e^{-K\tau}CD_{n-2}+(1-e^{-K\tau})D_n \\
& \leq e^{-K\tau}CD_{n-2}+(1-e^{-K\tau})D_{n-2} \\
& = \left( 1-e^{-K\tau}(1-C)\right)D_{n-2},
\end{split}
\end{equation*}
where $C$ is defined in \eqref{C}. So, setting 
$$ \tilde{C}:=1-e^{-K\tau}(1-C),$$
we have the claim. This implies that
\begin{equation} \label{ineq}
D_{3n} \leq \tilde{C}^n D_0, \ \forall n \geq 1.
\end{equation}
From \eqref{ineq}, we have that 
$$ D_{3n} \leq e^{-3n\gamma \tau}D_0, \ \forall n \in \N_0.$$
where 
$$ \gamma:=\frac{1}{3\tau}\ln \left(\frac{1}{\tilde{C}}\right).$$
Now, fix $i,j=1,\dots,N$ and $t\geq 0.$ Then one can find $t \in [3n\tau-\tau,3n\tau+2\tau]$ for some $n \in \N_0.$ Therefore, using Lemma \ref{2.2}, we find that 
$$ \vert x_i(t)-x_j(t)\vert \leq D_{3n} \leq e^{-3n\gamma \tau}D_0.$$
Thus, being $t \leq 3n\tau+2\tau,$ we get 
$$ \vert x_i(t)-x_j(t)\vert \leq e^{-\gamma t}e^{2\gamma \tau}D_0,$$
and, finally, we find that 
$$ d(t) \leq e^{-\gamma (t-2\tau)}D_0, \quad \forall t \geq 0.$$
So, \eqref{exp_decay} is proved.
\end{proof}

As a particular case of such a model, we can consider a system in which there is a (eventually small) group of  $m$ leaders, $m\in\N.$  These agents influence all other agents in the population, but they are influenced only by the other leaders. 
Let $y_i(t) \in \R^d, i=1,\dots, m,$ be the opinion of the i-th leader at time $t$ and $x_i(t) \in \R^d, \ i=1,\dots, N$ be  the opinion of the i-th non-leader at time $t.$ We assume  a non-universal interaction between the non-leaders. We still consider a time delay in the interaction between the agents to appear as a time needed to discuss and make a decision. Then, the opinions of the population evolve following  the Hegselmann-Krause opinion formation model: 

\begin{equation}\label{eqA}
\begin{array}{l}
\displaystyle{
 \frac{d}{d t}y_i(t)=  \frac{1}{m-1} \sum_{j \neq i}a_{ij}(t)(y_j(t-\tilde\tau_j)-y_i(t)), \quad t>0, \ i=1,...,m,} \\
\displaystyle{ \frac{d}{d t}x_i(t)=  \frac{1}{N+m-1} \sum_{ j \neq i } \chi_{ij}b_{ij}(t)(x_j(t-\tau_{ij})-x_i(t))}\\
\displaystyle{\hspace{3 cm} +  \frac{1}{N+m-1} \sum_{j=1}^{m} c_{ij}(t)(y_j(t-\tilde\tau_j)-x_i(t)),} \\
\displaystyle{ \hspace{7 cm} t>0,\ i=1,...,N,} 
\end{array}
\end{equation}
with the interaction weights $a_{ij}(t), b_{ij}(t), c_{ij}(t) \ t \ge 0,$   defined analogously to  \eqref{a}, namely
\begin{equation} \label{anew}
 a_{ij}(t):=\tilde\psi_{ij}(y_i(t),y_j(t-\tilde\tau_{j})), \quad \  i,j=1,...,m,
\end{equation}
\begin{equation} \label{b}
 b_{ij}(t):=\psi_{ij}(x_i(t),x_j(t-\tau_{ij})), \quad \  i,j=1,...,N,\hspace{1.7 cm}
\end{equation}
\begin{equation} \label{c}
 c_{ij}(t):=\psi^*_{ij}(x_i(t),y_j(t-\tilde\tau_{j})), \quad \  i=1,\dots,m,\ j=1,...,N,
\end{equation}
being the influence functions $\psi, \tilde\psi, \psi^*$ positive, continuous and bounded. 
Let us assume the initial conditions 
\begin{equation} \label{ic1A}
y_i(t)=y_i^0(t), \quad i=1,\dots,m, \ t \in [-\tau,0],
\end{equation}
and 
\begin{equation} \label{ic2A}
x_i(t)=x_i^0(t), \quad i=1,\dots,N, \ t \in [-\tau,0],
\end{equation}
where 
$$\tau=\max\left\{ \max_{i,j=1,\dots,N}{\tau_{ij}, \max_{j=1,\dots,m}\tilde\tau_j}\right\},$$
and
 $y_i^0(\cdot), i=1, \dots, m,$ and $x_i^0(\cdot), i=1, \dots N,$ are continuous functions.

\begin{Remark} \label{example_general}
    Note that system \eqref{eqA} satisfies the assumption \textbf{(CI)} if $m\ge 3.$ Then, in that case, the agents exponentially convergence to an asymptotic consensus.
\end{Remark}

The special cases with only one or two leaders deserve different analyses that will be developed in the next two sections. 

\section{A HK-model with one leader} \label{oneleader}
Here, we will focus on the case in which only one leader, influencing the other agents but not influenced by anyone, is present in the agents' group. First, we consider the leader to have a fixed opinion; then, we will deal with a controlled leader.
\subsection{A unique leader with constant trajectory} \label{one}
In this case, the system reads as: 
\begin{equation}\label{eqD}
\begin{split}
& \frac{d}{d t}y_0(t)= 0, \quad t>0, \\
& \frac{d}{d t}x_i(t)=  \frac{1}{N} \sum_{ j \neq i } \chi_{ij}b_{ij}(t)(x_j(t-\tau_{ij})-x_i(t)) +  \frac{c_{i0}(t)}N(y^0-x_i(t)), \\
& \hspace{7 cm} \quad t>0,\ i=1,...,N, 
\end{split}
\end{equation}
where $b_{ij}(t),$ $i,j=1,\dots,N,$ are defined as in \eqref{b} and $c_{0i}(t)= \phi_{0i}(x_i(t),y^0),$
with 
$\phi_{0i}:\R^d \times \R^d\rightarrow \R$  continuous, positive and bounded.
Let us assume the initial conditions
\begin{equation}\label{ICU}
y_0(0)=y^0 \quad \mbox{\rm and}\quad
x_i(t)= x_i^0(t), \quad t\in [-\tau, 0], \ \forall\ i=1, \dots, N,
\end{equation}
being, as before, $\tau=\max_{i,j=1,\dots,N} \tau_{ij}.$
We want to study the convergence to {consensus} of  system \eqref{eqD}. The \emph{diameter function} for solutions to \eqref{eqD} can be written as
$$d(t):=\max \left\{ \max_{i=1, \dots, N} |x_i(t)-y^0|, \max_{i,j=1, \dots, N}\vert x_i(t)-x_j(t)\vert \right\}.$$

Note that, in this case, assumption \textbf{(CI)} is not satisfied. Indeed, for each $i=1, \dots, N, $ the pair $(x_i, y_0)$ does not admit a common influencer. \\
Let us define the following quantities: 
\begin{equation} \label{Mone}
M_T:= \max \left\{ \max_{j=1,...,N} \max_{s \in [T-\tau,T]}\langle x_j(s), v \rangle, \langle y^0,v\rangle \right\},
\end{equation}
and 
\begin{equation} \label{mone}
m_T:= \min \left\{ \min_{j=1,...,N} \min_{s \in [T-\tau,T]}\langle x_j(s), v \rangle, \langle y^0,v\rangle \right\}.
\end{equation}
Moreover, let us define 
\begin{equation}\label{numero}
\tilde{K}:= \max \left\{ \max_{i,j=1,\dots,N}\Vert \psi_{ij} \Vert_{\infty},  \max_{i=1,\dots,N}\Vert\phi_{0i} \Vert_{\infty} \right\}.
\end{equation}
Similarly to the previous case, we can state the following results.

\begin{Lemma} \label{4.1}
Let $\Big (\{x_i(t)\}_{i=1}^N,y_0\Big )$ be a solution to \eqref{eqD}-\eqref{ICU}. Then, for each $v \in \R^d$ and $T \geq 0 $, we have 
\begin{equation} \label{eq4}
m_T \leq \langle x_i(t), v \rangle \leq M_T, 
\end{equation}
for all $t\ge T-\tau,$ $i=1, \dots, N.$ 
\end{Lemma}

\begin{Remark}
Notice that, from the definition \eqref{Mone} and \eqref{mone}, it is trivial that 
$$ m_T \leq \langle y^0,v \rangle \leq M_T.$$
\end{Remark}

Let us specify the notation in Definition \ref{def} for this case.
\begin{Definition} \label{defU}
We define 
\begin{equation}\label{D0U}
\begin{array}{l}
\displaystyle{ D_0:=\max \Big\{ \max_{i,j=1,\dots,N} \max_{s,t \in [-\tau,0]} \vert x_i(s)-x_j(t) \vert,} \\ 
\displaystyle{ \hspace{4 cm} \max_{i=1,\dots,N} \max_{s \in [-\tau,0]} \vert x_i(s)-y^0 \vert \Big \}, }
\end{array}
\end{equation}
and in general for all $n \in \N,$
\begin{equation}\label{Dn}
\begin{array}{l}
\displaystyle{ D_n:=\max \Big\{ \max_{i,j=1,\dots,N} \max_{s,t \in [n\tau-\tau,n\tau]} \vert x_i(s)-x_j(t) \vert,}  \\
\displaystyle{ \hspace{4 cm}\max_{i=1,\dots,N} \max_{s \in [n\tau-\tau,n\tau]} \vert x_i(s)-y^0 \vert \Big\}}.
\end{array}
\end{equation}
\end{Definition}
As before, we need preliminary estimates in order to prove the consensus result.
\begin{Lemma} \label{4.2}
Let $\Big (\{x_i(t)\}_{i=1}^N,y_0\Big )$ be a solution to \eqref{eqD}-\eqref{ICU}.
For each $n \in \N_0$ and $\forall \ i,j=1,...,N,$ we get
\begin{equation}\label{eq12U}
|x_i(t)-x_j(t)| \leq D_n, \quad \forall \ t \geq n\tau-\tau,\
\end{equation}
and, similarly, $\forall i=\dots,N,$
\begin{equation}\label{eq12U1}
|x_i(t)-y^0| \leq D_n, \quad \forall \ t \geq n\tau-\tau.\
\end{equation}
\end{Lemma}

Let us note that, as in Remark \ref{2.3}, Lemma \ref{4.2} implies that $\{ D_n \}_{n \in \N_0}$ is a non-increasing sequence.

\begin{Lemma} \label{4.4}
Let $\Big (\{x_i(t)\}_{i=1}^N,y_0\Big )$ be a solution to \eqref{eqD}-\eqref{ICU}.
For all $i=1,...,N,$ we have
\begin{equation}
|x_i(t)| \leq C_0, \quad \forall \ t \geq 0,
\label{eq16U}
\end{equation}
and, in particular,
\begin{equation}
|y^0| \leq C_0, \quad \forall \ t \geq 0,
\label{eq16U1}
\end{equation}
where
$$ C_0:=\max \left\{ \max_{i=1,...,N} \max_{s \in [-\tau,0]} |x_i(0)|, |y^0| \right\}.$$
\end{Lemma}

\begin{Remark} \label{4.5}
As in Remark \ref{2.5}, from Lemma \ref{4.4}, since the influence functions $\psi_{ij}$ and $\phi_{0i}$ are continuous and positive, we deduce that
\begin{equation} \label{lower_boundU}
\psi_{ij}(x_i(t),x_j(t-\tau_{ij})) \geq \psi_0:= \min_{i,j=1, \dots, N}\min_{|z_1|,|z_2| \leq C_0} \psi_{ij}(z_1,z_2) >0  , 
\end{equation}
for each $t \geq 0 $ and $i,j=1,...,N,$ and 
\begin{equation} \label{lower_boundU1}
\phi_{0i}(x_i(t),y^0) \geq \phi_0:= \min_{i=1, \dots, N}\min_{|z_1|,|z_2| \leq C_0} \phi_{0i}(z_1,z_2) >0  , 
\end{equation}
for each $t \geq 0 $ and $i=1,...,N$.
\end{Remark}
Let us denote
\begin{equation}\label{min_val}
\tilde{\psi}_0:=\min \{\psi_0, \phi_0\}.
\end{equation}

\begin{Lemma} \label{4.6}
Let $\Big (\{x_i(t)\}_{i=1}^N,y_0\Big )$ be a solution to \eqref{eqD}-\eqref{ICU}.
For all $i,j=1,\dots,N,$ unit vector $ v \in \R^d$ and $n \in \N_0$ we have that 
\begin{equation} \label{eq18one}
\langle x_i(t)-x_j(t),v \rangle \leq e^{-\tilde K(t-t_0)} \langle x_i(t_0)-x_j(t_0),v \rangle + (1-e^{-\tilde K(t-t_0)})D_n,
\end{equation}
for all $t \geq t_0 \geq n\tau, \ i,j=1,\dots,N,$ and 
\begin{equation}
\langle x_i(t)-y^0,v \rangle \leq e^{-\tilde K(t-t_0)} \langle x_i(t_0)-y^0,v \rangle + (1-e^{-\tilde K(t-t_0)})D_n,
\label{eq18U1}
\end{equation}
for all $t \geq t_0 \geq n\tau.$ 
Moreover, \eqref{eq20} holds true for all $n \in \N_0.$
\end{Lemma}
\begin{proof} The estimate \eqref{eq18one} can be proved analogously to \eqref{eq18}. So, let us prove \eqref{eq18U1}.
Let us fix $v \in \R^d$ such that $|v|=1$ and $n \in \N_0.$ We denote by
$$ M_n := \max \left\{ \max_{i=1,...,N} \max_{t \in [n\tau-\tau,n\tau]} \langle x_i(t), v \rangle, \langle y^0,v \rangle \right\}$$
and 
$$ m_n:= \min \left\{ \min_{i=1,...,N} \min_{t \in [n\tau-\tau, n\tau]} \langle x_i(t), v \rangle, \langle y^0,v \rangle \right\}.$$
Note that $M_n-m_n \leq D_n.$ Now, fix $i=1,\dots,N$ and $t \geq t_0 \geq n\tau.$ Then, as before, we have that
\begin{equation*}
\begin{split}
\frac{d}{d t} \langle x_i(t),v \rangle & = \frac{1}{N} \sum_{j \neq i} \chi_{ij} b_{ij}(t) \langle x_j(t-\tau_{ij})-x_i(t),v \rangle + \frac{c_{i0}(t)}{N}\langle y^0-x_i(t),v\rangle \\
& \leq  \frac{1}{N} (N-1)\tilde{K} \left( M_n -\langle x_i(t),v \rangle \right)
+ \frac{\tilde{K}}{N}\left(M_n-\langle x_i(t),v\rangle\right) \\
& \leq \tilde{K}(M_n - \langle x_i(t),v \rangle ).
\end{split}
\end{equation*}
By applying the Grönwall's Lemma, we find that 
\begin{equation}\label{eq21U}
\langle x_i(t),v \rangle \leq e^{-\tilde{K}(t-t_0)}  \langle x_i(t_0),v \rangle + (1-e^{-\tilde{K}(t-t_0)})M_n .
\end{equation}
On the other hand, notice that holds true that
\begin{equation*}
\frac{d}{d t} \langle y^0,v \rangle = 0\geq \tilde{K}(m_n - \langle y^0,v \rangle). 
\end{equation*}
Then, applying again the Grönwall's Lemma, 
\begin{equation}\label{eq22U}
\langle y^0,v \rangle \geq e^{-\tilde{K}(t-t_0)}  \langle y^0,v \rangle + (1-e^{-\tilde{K}(t-t_0)})m_n .
\end{equation}
From \eqref{eq21U} and \eqref{eq22U}, we have find \eqref{eq18U1}.
\end{proof}

Finally, we have the following result.
\begin{Lemma} \label{4.7}
Let $\Big (\{x_i(t)\}_{i=1}^N,y_0\Big )$ be a solution to \eqref{eqD}-\eqref{ICU}.
There exists a constant $C \in (0,1)$ such that 
\begin{equation} \label{estimate_diamU}
d(n\tau) \leq C D_{n-2},
\end{equation}
for all $n \geq 2$.
\end{Lemma}
\begin{proof}
Let us assume that $d(n\tau)= \vert x_i(n\tau)- y^0\vert$ for a given $i=1,\dots,N.$ We set 
    $$ v:=\frac{x_i(n\tau)-y^0}{|x_i(n\tau)-y^0|}.$$
Then, $d(n\tau)=\langle x_i(n\tau)-y^0, v \rangle$. Consider,
    $$ M_{n-1} := \max \left\{\max_{i=1,...,N} \max_{s \in [(n-2)\tau,(n-1)\tau]} \langle x_i(s), v \rangle, \langle y^0,v \rangle \right\} $$
            and 
    $$ m_{n-1}:= \min \left\{ \min_{i=1,...,N} \min_{s \in [(n-2)\tau,(n-1)\tau]} \langle x_i(s), v \rangle, \langle y^0,v \rangle \right\}. $$
It holds that $M_{n-1}-m_{n-1} \leq D_{n-1}$. \\
From \eqref{eqD}, for $t \in [(n-1)\tau,n\tau],$ we can write
\begin{equation*}
\begin{split}
& \frac{d}{d t} \langle x_i(t)-y^0, v \rangle  = \frac{1}{N} \sum_{j\neq i} \chi_{ij}b_{ij}(t) \langle x_j(t-\tau_{ij})-x_i(t), v \rangle + \frac{c_{i0}(t)}{N} \langle y^0-x_i(t), v \rangle \\
& = \frac{1}{N} \sum_{j\neq i} \chi_{ij}b_{ij}(t) \left(\langle x_j(t-\tau_{ij}),v \rangle - M_{n-1} \right) + \frac{1}{N} \sum_{j\neq i} \chi_{ij}b_{ij}(t) \left(M_{n-1}-x_i(t), v \rangle \right) \\
& \hspace{3.5 cm} + \frac{c_{i0}(t)}{N} \left(\langle y^0,v \rangle - M_{n-1}\right)  + \frac{c_{i0}(t)}{N} \left(M_{n-1}- x_i(t), v \rangle \right) \\
& \leq \frac{N-1}{N}\tilde{K}(M_{n-1}-\langle x_i(t), v \rangle) + \frac{\tilde{K}}{N}(M_{n-1}-\langle x_i(t), v \rangle)- \frac{\tilde{\psi_0}}{N} (\langle y^0, v \rangle - M_{n-1}). 
\end{split}
\end{equation*}
Then, 
\begin{equation}
\begin{split}
\frac{d}{d t} \langle x_i(t)-y^0, v \rangle  & \le \tilde{K}(M_{n-1}-\langle x_i(t), v \rangle)+ \frac{\tilde{\psi_0}}{N}(\langle y^0, v \rangle-M_{n-1}) \\
& \hspace{1.3 cm} + c_{i0}(t)(\langle y^0, v \rangle-m_{n-1}+m_{n-1}-\langle y^0, v \rangle) \\
& \leq \tilde{K}(M_{n-1}-\langle x_i(t), v \rangle)+ \frac{\tilde{\psi_0}}{N}(\langle y^0, v \rangle-M_{n-1}) \\
& \hspace{1.3 cm} + \tilde{K}(\langle y^0,v \rangle -m_{n-1})+\tilde{\psi_0}(m_{n-1}-\langle y^0,v\rangle) \\
& \leq \tilde{K}(M_{n-1}-\langle x_i(t), v \rangle)+ \frac{\tilde{\psi_0}}{N}(\langle y^0, v \rangle-M_{n-1}) \\
& \hspace{1.3 cm} + \tilde{K}(\langle y^0,v \rangle -m_{n-1})+\frac{\tilde{\psi_0}}{N}(m_{n-1}-\langle y^0,v\rangle) \\
& \leq \left(\tilde{K}-\frac{\tilde{\psi_0}}{N}\right)(M_{n-1}-m_{n-1})-\tilde{K}\langle x_i(t)-y^0,v \rangle,
\end{split}
\label{eq26U} 
\end{equation}
where we used the fact that $m_{n-1}-\langle y^0,v\rangle \leq 0.$
\\Notice that, repeating the argument in the proof of Lemma \ref{2.7}, from the estimate above, applying the Grönwall's Lemma, we can obtain a similar estimate as \eqref{2.37}. Finally, assuming that $d(n\tau)=|x_i(n\tau)-x_j(n\tau)|,$ for some $i,j=1,\dots,N,$ and using similar steps as above, we can find that 
$$ \frac{d}{d t} \langle x_i(t)-x_j(t),v \rangle \leq \left(\tilde{K}-\frac{\tilde{\psi_0}}{N}\right)(M_{n-1}-m_{n-1})-\tilde{K}\langle x_i(t)-x_j(t),v \rangle.$$ 
Therefore, we have the result with 
\begin{equation} \label{CU}
C:= 1-\frac{\tilde{\psi_0}}{\tilde{K}N}(1-e^{-\tilde{K}\tau}).
\end{equation}
\end{proof}
Then, the exponential consensus result follows arguing as in previous section.

\begin{Theorem}\label{main_1leaderfermo}
   Every solution
 $\Big (\{x_i(t)\}_{i=1}^N,y_0\Big )$  to \eqref{eqD}-\eqref{ICU}
 satisfies the exponential decay estimate
    \begin{equation} \label{exp_decay_1leaderfermo}
        d(t)\leq {D_0}e^{-\tilde\gamma (t-2\tau)},\quad \forall \ t \geq 0,
    \end{equation}
    for a suitable positive constant $\tilde\gamma.$
\end{Theorem}

\subsection{A unique leader with controlled trajectory}
Consider now a HK-model with a unique leader following a controlled trajectory: 
\begin{equation}\label{eqcontr}
\begin{split}
& \frac{d}{d t}y_0(t)= u(t) \in \R^d, \quad t>0, \\
& \frac{d}{d t}x_i(t)=  \frac{1}{N} \sum_{ j \neq i } \chi_{ij}b_{ij}(t)(x_j(t-\tau_{ij})-x_i(t)) +  \frac{c_{i0}(t)}N(y_0(t-\tau_{i0})-x_i(t)), \\
& \hspace{7 cm} \quad t>0,\ i=1,...,N, 
\end{split}
\end{equation}
where $b_{ij}(t),$ $i,j=1,\dots,N,$ are defined as in \eqref{b} and
 $c_{0i}(t)= \phi_{0i}(x_i(t),y_0(t)),$
with 
$\phi_{0i}:\R^d \times \R^d\rightarrow \R$  continuous, positive and bounded, as before.
Let us assume the initial conditions
\begin{equation}\label{ICcont}
y_0(0)=y^0(t) \quad \mbox{\rm and}\quad
x_i(t)= x_i^0(t), \quad t\in [-\tau, 0], \ \forall\ i=1, \dots, N,
\end{equation}
being, as before, $\tau=\max \left\{ \max_{i,j=1,\dots,N} \tau_{ij}, \max_{i=1,\dots,N} \tau_{i0} \right\}.$
Here, we want to study the convergence to any consensus state for solutions to  system \eqref{eqcontr}. In this case, the  definition of the \emph{diameter function} can be written as
$$d(t):=\max \left\{ \max_{i=1, \dots, N} |x_i(t)-y_0(t)|, \max_{i,j=1, \dots, N}\vert x_i(t)-x_j(t)\vert \right\}.$$

Note that, also in this configuration, assumption \textbf{(CI)} is not satisfied.\\
Consider, again, the appropriate definitions of $m_T$ and $M_T$: 
$$ M_T:= \max \left\{ \max_{j=1,...,N} \max_{s \in [T-\tau,T]}\langle x_j(s), v \rangle, \max_{s \in [T-\tau,T]} \langle y_0(s),v\rangle \right\},$$
$$ m_T:= \min \left\{ \min_{j=1,...,N} \min_{s \in [T-\tau,T]}\langle x_j(s), v \rangle, \min_{s \in [T-\tau,T]} \langle y_0(s),v\rangle \right\}.$$
and let
$\tilde{K}$ be the constant defined in \eqref{numero}.
\begin{Definition}
Fix a constant $M>0.$ We say that a ${\mathcal L^1}$-measurable strategy $u(\cdot)$  is admissible if $\Vert u \Vert_{\infty} \leq M.$
\end{Definition}

We will prove that, given any state $\bar{x} \in \R^d,$ there exists an admissible strategy that steers each agent to $\bar{x}.$ Firstly, in the following lemma we prove that there exists an admissible strategy that steers the leader to $\bar{x} \in \R^d$ and the other agents stay close enough to $\bar{x}.$
\begin{Lemma}
Let us consider the  system \eqref{eqcontr}-\eqref{ICcont} and let $\bar{x} \in \R^d.$ Then, there exists an admissible control strategy $u:[0,+\infty)\rightarrow\R^d$ such that $y_0$ reaches $\bar{x}$ in finite time, i.e. there exists $t_0>0$ such that $y_0(t)=\bar{x}$ for all $t\geq t_0.$ Moreover, if we denote
\begin{equation} \label{Rdef}
R:=\max \left\{\max_{i=1,\dots,N} \max_{s \in [-\tau,0]} \vert x_i(s)-\bar{x} \vert, \max_{s\in [-\tau, 0]}\vert y_0(s)-\bar{x}\vert \right\},
\end{equation}
then 
\begin{equation} \label{close}
\max_{i=1,\dots,N} \vert x_i(t)-\bar{x}\vert \leq R,
\end{equation}
for all $t\geq0.$
\end{Lemma}
\begin{proof}
Let us define for $t\geq0$ the admissible control 
\begin{equation} \label{u}
u(t):=
\begin{cases}
M\frac{\bar{x}-y^0}{\vert \bar{x}-y^0\vert}, \quad\ \  \mbox{ if } y_0(t)\neq\bar{x}, \\
0, \hspace{2 cm} \mbox{ if } y_0(t)=\bar{x},
\end{cases}
\end{equation}
where $y^0=y_0(0).$
Consider the unit vector $v \in \R^d$ defined as 
$$ v:=\frac{y^0-\bar{x}}{\vert y^0-\bar{x}\vert},$$
so that $\langle y^0-\bar{x},v \rangle = \vert y^0-\bar{x} \vert.$
Therefore, from the definition of the system \eqref{eqcontr} and \eqref{u}, we have for $t\geq0$ that
\begin{equation*}
\frac{d}{d t} \langle y_0(t)-\bar{x},v \rangle = \frac{M}{\vert \bar{x}-y^0 \vert} \langle \bar{x}-y^0,v \rangle  = - \frac{M}{\vert \bar{x}-y^0\vert} \vert \bar{x}-y^0\vert = -M.
\end{equation*}
Therefore, 
$$ \langle y_0(t)-\bar{x},v \rangle \leq \vert y_0(t)-\bar{x} \vert \leq \vert y^0-\bar{x}\vert -Mt, \quad t \in [0,t_0],$$
where 
$$ t_0:=\frac{\vert y^0-\bar{x}\vert}{M}.$$
It is immediate to see that 
$$\langle y_0(t)-\bar{x},v \rangle=\vert y_0(t)-\bar{x}\vert,$$
then the first part of the lemma is proven. Notice now that by definition $\vert y^0-\bar{x}\vert \leq R,$ then $$\vert y_0(t)-\bar{x}\vert \leq R,$$ for all $t\geq0.$ Now, let $v \in \R^d$ be a unit vector and let us define the set 
$$ \mathcal{K}_{\epsilon}:=\left\{ t \in [0,+\infty) \Big\vert \max_{i=1,\dots,N} \langle x_i(s)-\bar{x},v \rangle < R+\epsilon, \ \forall s \in [0,t) \right\}.$$
By continuity, $\mathcal{K}\neq \emptyset.$ Let $S:=\sup \mathcal{K}_{\epsilon}.$ We claim that $S=+\infty.$ In order to prove it, suppose by contraddiction that $S<+\infty.$ Then,
$$ \lim_{t \rightarrow S^-} \max_{i=1,\dots,N} \langle x_i(s)-\bar{x},v \rangle = R+\epsilon.$$ 
Consider $t \in [0,S).$ Then, from \eqref{eqcontr}, we have that
\begin{equation*}
\begin{split}
\frac{d}{d t} \langle x_i(t)-\bar{x},v \rangle & = \frac{1}{N}\sum_{j \neq i} \chi_{ij}b_{ij}(t) \langle x_j(t-\tau_{ij})-x_i(t),v \rangle \\
& \hspace{2.5 cm}+ \frac{c_{i0}(t)}{N} \langle y_0(t-\tau_{i0})-x_i(t),v \rangle \\
& = \frac{1}{N}\sum_{j \neq i} \chi_{ij}b_{ij}(t) \left(\langle x_j(t-\tau_{ij})-\bar{x},v \rangle - \langle x_i(t)-\bar{x},v \rangle \right) \\
& \hspace{2.5 cm} + \frac{c_{i0}(t)}{N} \left(\langle y_0(t-\tau_{i0})-\bar{x},v \rangle - \langle x_i(t)-\bar{x},v \rangle \right) \\
& \leq \tilde{K}(R+\epsilon-\langle x_i(t)-\bar{x},v\rangle),
\end{split}
\end{equation*} 
where we used the fact that, if $t \in [0,S),$ then $t-\tau_{ij} \in [-\tau,S),$ for all $i=1,\dots,N, \ j=0,\dots,N,$ and, by definition of $\mathcal{K}_\epsilon$ we have 
$$  \langle x_j(t-\tau_{ij})-\bar{x},v \rangle \leq R+\epsilon,$$
and
$$ \langle y_0(t-\tau_{i0})-\bar{x},v \rangle \leq R+\epsilon.$$
Using the Grönwall inequality, we get 
\begin{equation*}
\begin{split}
 \langle x_i(t)-\bar{x},v \rangle & \leq e^{-\tilde{K}t} \langle x_i(0)-\bar{x},v \rangle+(R+\epsilon)(1-e^{-\tilde{K}t}) \\
 & \leq e^{-\tilde{K}t}R+R+\epsilon-e^{-\tilde{K}t}R -\epsilon e^{-\tilde{K}t} \\
& = R+\epsilon - \epsilon e^{-\tilde{K}t} \\
& \leq R+\epsilon - \epsilon e^{-\tilde{K}S} < R+\epsilon.
\end{split}
\end{equation*}
Therefore, sending $t \rightarrow S^-,$ we finally find 
$$  \lim_{t \rightarrow S^-} \max_{i=1,\dots,N} \langle x_i(t)-\bar{x},v \rangle < R +\epsilon,$$
that gives the contradiction. Thus, $S=+\infty$ and 
$$ \max_{i=1,\dots,N} \langle x_i(t)-\bar{x},v \rangle < R + \epsilon,$$
for all $t\geq 0$ and $v \in \R^d.$ Because of the arbitrarily of $v \in \R^d$ we can choose 
$$ v:=\frac{x_i(t)-\bar{x}}{\vert x_i(t)-\bar{x} \vert},$$ 
so that $  \langle x_i(t)-\bar{x},v \rangle = \vert x_i(t)-\bar{x} \vert.$ Hence, we have that 
$$ \max_{i=1,\dots,N} \vert x_i(t)-\bar{x}\vert < R + \epsilon.$$
Finally, sending $\epsilon \rightarrow 0,$ we have the result. 
\end{proof}

We are now able to prove the following result.
\begin{Theorem}\label{maincont}
Let us consider the  system \eqref{eqcontr}-\eqref{ICcont} and let $\bar{x} \in \R^d.$ Then, there exists an admissible control strategy $u:[0,+\infty)\rightarrow\R^d$ such that the leader $y_0$ reaches $\bar{x}$ in finite time, and
    \begin{equation} \label{exp_decay_control}
        d(t)\leq {C^*}e^{-{\gamma}^* t}, \quad \forall \  t \geq t_0,
    \end{equation}
for suitable positive constants ${\gamma}^*$ and $C^*.$
\end{Theorem}
\begin{proof}
Notice that with the preliminary lemma we prove that the leader reaches in a finite time $t_0$ the state $\bar{x} \in \R^d$ and stay there for all $t\geq ,t_0,$ and the other agents maintain the distance from $\bar{x}$ less or equal than $R.$ Now, in order to prove the consensus result one can apply the results in Section \ref{one}, since 
$$ \frac{d}{d t} y_0(t)=0,$$
for all $t\geq t_0.$
\end{proof}

\begin{Remark}
Note that the above result extends the one in \cite{WCB}, that proved the consensus in the absence of time delays, and the one in \cite{PaolucciP}, considering delay effects, since now we do not need any restrictions on the time delay size. We point that \cite{PaolucciP,WCB} assume the  coefficients  $b_{ij}(t)$ compactly supported while here they are always positive. However, it is easy to see from the proof that our result is valid also in the case $b_{ij}(t)$ compactly supported. What we need is, indeed, that the leader's influence is always positive only, exactly as in \cite{PaolucciP,WCB}.
\end{Remark}

\section{A HK-model with two leaders} \label{twoleaders}
Now, we want to analyze the case of a Hegselmann-Krause model in the presence of two leaders. The system reads as follows:
\begin{equation}\label{eqDUE}
\begin{array}{l}
\displaystyle{
 \frac{d}{d t}y_i(t)=  a_{ij}(t)(y_j(t-\tilde\tau_j)-y_i(t)), \quad t>0, \ i=1,2,} \\
\displaystyle{ \frac{d}{d t}x_i(t)=  \frac{1}{N+1} \sum_{ j \neq i } \chi_{ij}b_{ij}(t)(x_j(t-\tau_{ij})-x_i(t))}\\
\displaystyle{\hspace{1.3 cm} +  \frac{1}{N+1} \sum_{j=1}^{2} c_{ij}(t)(y_j(t-\tilde\tau_j)-x_i(t)), \quad t>0,\ i=1,...,N,} 
\end{array}
\end{equation}
with the interaction weights $a_{ij}(t), b_{ij}(t), c_{ij}(t), \ t \ge 0,$   defined as in   \eqref{anew}, \eqref{b}, \eqref{c}. \\
Also in this case the assumption \textbf{(CI)} is not satisfied. Indeed, the two leaders do not admit a common influencer. 
Let us assume the initial conditions
\begin{equation}\label{IC2}
\begin{split}
& y_i(t)=y_i^0(t), \  \quad t \in [-\tau,0], \ i=1,2, \\
& x_i(t)= x_i^0(t), \quad t\in [-\tau, 0], \ \forall\ i=1, \dots, N,
\end{split}
\end{equation}
being $\tau=\max \left\{\tilde\tau_1, \tilde\tau_2, \max_{i,j=1,\dots,N} \tau_{ij}\right \}.$
We want to study the convergence to {consensus} of  system \eqref{eqDUE}. The \emph{diameter function} can be written, in this case, as
$$d(t):=\max \left\{ \max_{\substack{i=1, \dots, N \\ j=1,2}} |x_i(t)-y_i(t)|, \max_{i,j=1, \dots, N}\vert x_i(t)-x_j(t)\vert, \vert y_1(t)-y_2(t)\vert \right\}.$$

Again, let us define the following quantities: 
$$ M_T:= \max \left\{ \max_{j=1,...,N} \max_{s \in [T-\tau,T]}\langle x_j(s), v \rangle, \max_{i=1,2} \max_{t \in [-\tau,0]} \langle y_i(t),v\rangle \right\},$$
and 
$$ m_T:= \min \left\{ \min_{j=1,...,N} \min_{s \in [T-\tau,T]}\langle x_j(s), v \rangle, \min_{i=1,2} \min_{t \in [-\tau,0]} \langle y_i(t),v\rangle \right\},$$
and let  
$$\tilde K:=\max\left\{  \max_{i=1,2}\Vert \tilde\psi_{ij}\Vert_{\infty}, \max_{i,j=1,\dots,N}\Vert\psi_{ij}\Vert_\infty, \max_{i=1,...,N} \Vert\psi^*_{i1}\Vert_\infty, \max_{i=1,...,N} \Vert\psi^*_{i2}\Vert_\infty \right\}.$$

Analogously to the one leader case, we can prove the following preliminary lemmas.

\begin{Lemma} \label{2.1doppione}
Let $\left(\{x_i(t)\}_{i=1}^N,\{y_j(t)\}_{j=1}^2\right)$ be a solution to \eqref{eqDUE}-\eqref{IC2}. Then, for each $v \in \R^d$ and $T \geq 0 $, we have 
\begin{equation} \label{eq4.2}
m_T  \leq \langle x_i(t), v \rangle \leq M_T, 
\end{equation}
and 
\begin{equation} \label{eq4.2}
m_T  \leq \langle y_j(t), v \rangle \leq M_T, 
\end{equation}
for all $t\ge T-\tau,$ $i=1, \dots, N$ and $j=1,2.$
\end{Lemma}
Let us introduce the appropriate notation:
\begin{Definition} \label{def2}
We define 
\begin{equation}\label{D02}
\begin{split}
& D_0:=\max \Big\{ \max_{i,j=1,\dots,N} \max_{s,t \in [-\tau,0]} \vert x_i(s)-x_j(t) \vert, \max_{\substack{i=1,\dots,N \\ j=1,2}} \max_{s,t \in [-\tau,0]} \vert x_i(s)-y_j(t) \vert, \\
& \hspace{5.5 cm } \max_{s,t \in [-\tau,0]} \vert y_1(s)-y_2(t) \vert  \Big\},
\end{split}
\end{equation}
and in general for all $n \in \N,$
\begin{equation}\label{Dn2}
\begin{split}
& D_n:=\max \Big\{ \max_{i,j=1,\dots,N} \max_{s,t \in [n\tau-\tau,n\tau]} \vert x_i(s)-x_j(t) \vert, \\
&  \hspace{2 cm} \max_{\substack{i=1,\dots,N \\ j=1,2}} \max_{s,t \in [n\tau-\tau,n\tau]} \vert x_i(s)-y_j(t) \vert, 
\max_{s,t \in [n\tau-\tau,n\tau]} \vert y_1(s)-y_2(t) \vert \Big\}.
\end{split}
\end{equation}
\end{Definition}

\begin{Lemma} \label{5.2}
Let $\left(\{x_i(t)\}_{i=1}^N,\{y_j(t)\}_{j=1}^2\right)$ be a solution to \eqref{eqDUE}-\eqref{IC2}. 
For each $n \in \N_0,$  we get
\begin{equation}\label{eq12.2}
\begin{split}
& |x_i(t)-x_j(t)| \leq D_n, \quad \forall \ i,j=1,...,N,\\
& |x_i(t)-y_j(t)| \leq D_n, \quad \forall \ i=1,\dots,N, \ j=1,2,\\
& |y_1(t)-y_2(t)| \leq D_n,
\end{split}
\end{equation}
for all $ t \geq n\tau-\tau.$
\end{Lemma}
Again, as a consequence of Lemma \ref{5.2} we have that $\{D_n\}_{n \in \N_0}$ is a non-increasing sequence. 

\begin{Lemma} \label{5.4}
Let $\left(\{x_i(t)\}_{i=1}^N,\{y_j(t)\}_{j=1}^2\right)$ be a solution to \eqref{eqDUE}-\eqref{IC2}. Then,
for all $i=1,...,N$ and $j=1,2,$ we have
\begin{equation}
\begin{split}
& |x_i(t)| \leq C_0, \\
& |y_j(t)| \leq C_0
\end{split}
\label{eq16.2}
\end{equation}
for all $t \geq 0,$ where
$$ C_0:=\max \left\{ \max_{i=1,...,N} \max_{s \in [-\tau,0]} |x_i(0)|, \max_{i=1,2} \max_{s \in [-\tau,0]} |y_i(0)| \right\}.$$
\end{Lemma}
From Lemma \ref{5.4} we deduce again a positive lower bound, $\tilde\psi_0,$ as in \eqref{lower_bound}, for the interaction rates.
Moreover, we can prove the following estimates.
\begin{Lemma} \label{5.6}
Let $\left(\{x_i(t)\}_{i=1}^N,\{y_j(t)\}_{j=1}^2\right)$ be a solution to \eqref{eqDUE}-\eqref{IC2}. 
For any unit vector $ v \in \R^d$ and $n \in \N_0$ we have that 
\begin{equation} 
\begin{split}
& \langle x_i(t)-x_j(t),v \rangle \leq e^{-\tilde K(t-t_0)} \langle x_i(t_0)-x_j(t_0),v \rangle + (1-e^{-\tilde K(t-t_0)})D_n, \\
& \hspace{6 cm}\quad i,j=1,\dots,N \\
& \langle x_i(t)-y_j(t),v \rangle \leq e^{-\tilde K(t-t_0)} \langle x_i(t_0)-y_j(t_0),v \rangle + (1-e^{-\tilde K(t-t_0)})D_n, \\
& \hspace{5.6 cm}\quad i=1,\dots,N, \ j=1,2, \\
& \langle y_1(t)-y_2(t),v \rangle \leq e^{-\tilde K(t-t_0)} \langle y_1(t_0)-y_2(t_0),v \rangle + (1-e^{-\tilde K(t-t_0)})D_n,
\end{split}
\label{eq18U}
\end{equation} 
for all $t \geq t_0 \geq n\tau.$ Moreover, \eqref{eq20} holds true for all $n \in \N_0.$
\end{Lemma}
\begin{proof}
Notice that the first two inequalities in \eqref{eq18U} can be obtained arguing as in Section \ref{Prel}. For the last inequality in \eqref{eq18U}, we want to emphasize that the equations for the leaders $y_1(t)$ and $y_2(t)$ are independent from the other agents $\{x_i(t)\}_{i=1}^{N}.$ Therefore, it represents an Hegselmann-Krause $2 \times 2$ system itself. So, the estimate follows as in \cite{CPM2A}.
\end{proof}

\begin{Lemma} \label{5.7}
Let $\left(\{x_i(t)\}_{i=1}^N,\{y_j(t)\}_{j=1}^2\right)$ be a solution to \eqref{eqDUE}-\eqref{IC2}. 
Then, there exists a constant $C \in (0,1)$ such that 
\begin{equation} \label{estimate_diam_leader}
d(n\tau) \leq C D_{n-2},
\end{equation}
for all $n \geq 2$.
\end{Lemma}
\begin{proof}
Let us suppose, as before, without loss of generality that $d(n \tau)>0$. 
Let us assume that $d(n \tau)= |x_i(n \tau)-y_j(n \tau)|,$ for some $i=1,\dots,N$ and $j=1,2.$ In particular, suppose $d(n \tau)= |x_i(n \tau)-y_1(n \tau)|.$ Let us define the unit vector $v \in \R^d$ as 
$$ v:=\frac{x_i(n \tau)-y_1(n \tau)}{\vert x_i(n \tau)-y_1(n \tau)\vert}.$$
Then we have that $d(n\tau)=\langle x_i(n \tau)-y_1(n \tau),v \rangle.$
Consider, as before, 
$$ M_{n-1}= \max \left\{ \max_{l=1,...,N} \max_{s \in [(n-2)\tau, (n-1)\tau]} \langle x_l(s),v \rangle, \max_{l=1,2} \max_{s \in [(n-2)\tau, (n-1)\tau]} \langle y_l(s),v\rangle \right\}  $$
and 
$$ m_{n-1}= \min \left\{ \min_{l=1,...,N} \min_{s \in [(n-2)\tau,(n-1)\tau]}  \langle x_l(s),v \rangle, \min_{l=1,2} \min_{s \in [(n-2)\tau, (n-1)\tau]} \langle y_l(s),v \rangle \right\} . $$
Again, we have that $M_{n-1}-m_{n-1}\leq D_{n-1}$.
For $t \in [(n-1)\tau, n \tau]$ we have that 
\begin{equation}  \label{eq}
\begin{split}
& \frac{d}{d t}\langle x_i(t)- y_1(t),v \rangle = \frac{1}{N+1} \sum_{k \neq i }\chi_{ik}b_{ik}(t)\langle x_k(t-\tau_{ik})-x_i(t),v \rangle \\
& \hspace{0.6 cm} +\frac{1}{N+1}\sum_{k=1}^2 c_{ik}(t)\langle y_k(t-\tilde\tau_k)-x_i(t),v \rangle - a_{12}(t)\langle y_2(t-\tilde\tau_2)-y_1(t),v \rangle \\
& = \frac{1}{N+1} \sum_{k \neq i }\chi_{ik}b_{ik}(t)(\langle x_k(t-\tau_{ik}),v \rangle-M_{n-1}+M_{n-1}- \langle x_i(t),v \rangle) \\
& \hspace{1 cm} + \frac{1}{N+1} \sum_{k=1}^2 c_{ik}(t)(\langle y_k(t-\tilde\tau_k),v \rangle -M_{n-1}+M_{n-1}-\langle x_i(t),v \rangle) \\
&\hspace{1 cm}  -a_{12}(t)(\langle y_2(t-\tilde\tau_2),v \rangle -m_{n-1}+m_{n-1}-\langle y_1(t),v \rangle) \\
& =: S_1+S_2.
\end{split}
\end{equation}
Since $t \in [(n-1)\tau,n \tau]$, then $ t-\tilde\tau_{k}, t-\tau_{ij} \in [(n-2)\tau,n\tau]$ for all $k=1,2,$ $i,j=1, \dots,N.$ Then, we can write
\begin{equation} 
\begin{split}
& S_1 = \frac{1}{N+1} \sum_{k \neq i } \chi_{ik}b_{ik}(t)(\langle x_k(t-\tau_{ik}),v \rangle-M_{n-1})\\
& \hspace{0.8 cm} + \frac{1}{N+1} \sum_{k \neq i } \chi_{ik}b_{ik}(t)(M_{n-1}-\langle x_i(t),v \rangle) \\
& \hspace{0.8 cm} + \frac{1}{N+1} \sum_{k=1}^2 c_{ik}(t)(\langle y_k(t-\tilde\tau_k),v \rangle -M_{n-1}) \\
& \hspace{0.8 cm} + \frac{1}{N+1} \sum_{k=1}^2 c_{ik}(t)(M_{n-1} - \langle x_i(t),v \rangle). 
\end{split}
\end{equation}
So, we find that
\begin{equation} \label{s1}
\begin{split}
& S_1 \leq \frac{\tilde\psi_0}{N+1}  \sum_{k \neq i} \chi_{ik} (\langle x_k(t-\tau_{ik}),v \rangle-M_{n-1}) + \frac{N-1}{N+1}\tilde K(M_{n-1}-\langle x_i(t),v \rangle) \\
& \hspace{0.8 cm} + \frac{\tilde \psi_0}{N+1}\sum_{k=1}^2(\langle y_k(t-\tilde\tau_k),v \rangle -M_{n-1})+ \frac{2}{N+1}\tilde K(M_{n-1} - \langle x_i(t),v \rangle) \\
& = \tilde K(M_{n-1} - \langle x_i(t),v \rangle)+\frac{\tilde\psi_0}{N+1}  \sum_{k \neq i} \chi_{ik} (\langle x_k(t-\tau_{ik}),v \rangle-M_{n-1}) \\
& \hspace{0.8 cm} + \frac{\tilde \psi_0}{N+1}\sum_{k=1}^2(\langle y_k(t-\tilde\tau_k),v \rangle -M_{n-1}).
\end{split}
\end{equation}
Analogously, we can estimate
\begin{align}
\begin{aligned}
& S_2:= a_{12}(t)(m_{n-1}-\langle y_2(t-\tilde\tau_2),v \rangle + a_{12}(t)(\langle y_1(t),v \rangle-m_{n-1})\cr
& \leq \tilde\psi_0 (m_{n-1}-\langle y_2(t-\tilde\tau_2),v \rangle + \tilde K(\langle y_1(t),v \rangle-m_{n-1}).
\end{aligned} \label{s2}
\end{align}
Putting \eqref{s1} and \eqref{s2} in \eqref{eq}, we can write
\begin{align}
\begin{aligned}
& \frac{d}{d t} \langle x_i(t)-y_1(t),v \rangle \leq \tilde K(M_{n-1}-m_{n-1}) -\tilde K \langle x_i(t)-y_1(t),v \rangle \cr
& + \frac{\tilde \psi_0}{N+1}  \sum_{k \neq i} \chi_{ik} (\langle x_k(t-\tau_{ik}),v \rangle-M_{n-1})+ \frac{\tilde \psi_0}{N+1}\sum_{k=1}^2(\langle y_k(t-\tilde\tau_k),v \rangle -M_{n-1})  \cr
& \hspace{5cm}+ \tilde\psi_0 (m_{n-1}-\langle y_2(t-\tilde\tau_2),v \rangle \cr
& \leq  \tilde K(M_{n-1}-m_{n-1}) -\tilde K \langle x_i(t)-y_1(t),v \rangle \cr
& \hspace{1 cm}+\frac{\tilde\psi_0}{N+1}\sum_{k=1}^2(\langle y_k(t-\tilde\tau_k),v \rangle -M_{n-1})+ \tilde \psi_0 (m_{n-1}-\langle y_2(t-\tilde\tau_2),v \rangle ).
\end{aligned}
\end{align}
Notice that, since $\langle y_k(t-\tilde\tau_k),v \rangle -M_{n-1} \leq 0$ for $k=1,2$ and $t \in [(n-1)\tau, n\tau],$ we have that
$$ \sum_{k=1}^2 (\langle y_k(t-\tilde\tau_k),v \rangle -M_{n-1}) \leq \langle y_2(t-\tilde\tau_2),v \rangle -M_{n-1}. $$
Therefore, we have that 
\begin{equation}
\begin{split}
& \frac{d}{d t} \langle x_i(t)-y_1(t),v \rangle \leq \tilde K(M_{n-1}-m_{n-1}) - \tilde K \langle x_i(t)-y_1(t),v \rangle \\
&  + \frac{\tilde\psi_0}{N+1}(\langle y_2(t-\tilde\tau_2),v \rangle -M_{n-1}+m_{n-1}-\langle y_2(t-\tilde\tau_2),v \rangle ) \\
& =\left( \tilde K-\frac{\tilde\psi_0}{N+1}\right)(M_{n-1}-m_{n-1})- \tilde K \langle x_i(t)-y_1(t),v \rangle,
\end{split}
\end{equation}
where we used that $\tilde \psi_0(m_{n-1}-\langle y_2(t-\tilde\tau_2),v \rangle ) \leq   \frac{\tilde\psi_0}{N+1}(m_{n-1}-\langle y_2(t-\tilde\tau_2),v \rangle ) $. \\
Applying now the Grönwall inequality over $[(n-1)\tau, t],$ with $t \in ((n-1)\tau,n\tau],$ we have
\begin{equation}
\begin{split}
& \langle x_i(t)-y_1(t),v \rangle \leq e^{-\tilde K(t-(n-1)\tau)}\langle x_i((n-1)\tau)-y_1((n-1)\tau),v \rangle \\
& \hspace{1 cm}+ \left( 1- \frac{\tilde\psi_0}{\tilde K(N+1)} \right) (M_{n-1}-m_{n-1})(1-e^{-\tilde K(t-(n-1)\tau)}).
\end{split}
\end{equation}
For $t=n\tau,$ we have 
\begin{equation}
\begin{split}
& d(n\tau)\leq e^{-\tilde K\tau}\langle x_i((n-1)\tau)-y_1((n-1)\tau),v \rangle \\
&\hspace{1.2 cm}+\left( 1- \frac{\tilde\psi_0}{\tilde K(N+1)} \right) (M_{n-1}-m_{n-1})(1-e^{-\tilde K\tau}) \\
& \hspace{1 cm}\leq \left( 1- \frac{\tilde \psi_0}{\tilde K(N+1)}(1-e^{-\tilde K\tau})\right)D_{n-2},
\end{split}
\end{equation}
where we used that $M_{n-1}-m_{n-1}\leq D_{n-1}$ and the monotonicity property of $D_n.$ \\
Notice that if $d(n\tau)=|x_i(n\tau)-x_j(n\tau)|,$ for some fixed $i,j=1,\dots,N,$ using a similar argument, we can find again that 
$$ \frac{d}{d t} \langle x_i(t)-x_j(t),v \rangle \leq \left( \tilde K- \frac{\tilde \psi_0}{N+1}\right)(M_{n-1}-m_{n-1})-\tilde K\langle x_i(t)-x_j(t),v\rangle,$$
for $t \in [(n-1)\tau,n\tau].$ \\
Finally, assume that $d(n\tau)=|y_1(n\tau)-y_2(n\tau)|.$ Let us define again an unit vector $v \in \R^d$ as 
$$ v:=\frac{y_1(n\tau)-y_2(n\tau)}{|y_1(n\tau)-y_2(n\tau)|}.$$
Then, we can write $d(n\tau)=\langle y_1(n\tau)-y_2(n\tau),v\rangle.$ Consider $t \in [(n-2)\tau, n\tau].$ Let us distinguish two cases.\\
\textbf{Case 1.} There exists $t_0 \in [(n-2)\tau,n\tau]$ such that $\langle y_1(t_0)-y_2(t_0),v \rangle<0.$ Applying Lemma \ref{5.6}, we get 
\begin{equation}
\begin{split}
d(n\tau) & \leq e^{-\tilde K(n\tau-t_0)} \langle y_1(t_0)-y_2(t_0),v \rangle + (1-e^{-\tilde K(n\tau-t_0)})D_{n-2} \\
& \leq (1-e^{-\tilde K(n\tau-t_0)})D_{n-2} \\
& \leq (1-e^{-2\tilde K\tau})D_{n-2},
\end{split}
\end{equation}
and we have the statement.
\\\textbf{Case 2.} Assume that $\langle y_1(t)-y_2(t),v \rangle \geq 0, \ \forall t \in [(n-2)\tau, n\tau].$ Consider $t \in [(n-1)\tau,n\tau].$ From \eqref{eqDUE} we have that 
\begin{equation*}
\begin{split}
\frac{d}{d t} \langle  y_1(t)-y_2(t) & ,v  \rangle  = a_{12}(t)\langle y_2(t-\tilde\tau_2)-y_1(t),v \rangle - a_{21}(t)\langle y_1(t-\tilde\tau_1)-y_2(t),v \rangle \\
& = a_{12}(t)\left(\langle y_2(t-\tilde\tau_2,v \rangle -M_{n-1}\right)+a_{12}(t)\left(M_{n-1}-\langle y_1(t),v \rangle \right) \\
& + a_{21}(t)\left(m_{n-1}-\langle y_1(t-\tilde\tau_1),v\rangle \right)+ a_{21}(t)\left(\langle y_2(t),v \rangle - m_{n-1}\right) \\
& \leq \tilde \psi_0 \left(\langle y_2(t-\tilde\tau_2),v \rangle -M_{n-1}\right)+\tilde K \left(M_{n-1}-\langle y_1(t),v \rangle \right) \\
& + \tilde\psi_0 \left(m_{n-1}-\langle y_1(t-\tilde\tau_1),v\rangle \right)+ \tilde K\left(\langle y_2(t),v \rangle - m_{n-1}\right). 
\end{split}
\end{equation*}
Then, 
\begin{equation*}
\begin{split}
\frac{d}{d t} \langle y_1(t)-y_2(t),v \rangle & = \left(\tilde K-\tilde \psi_0\right)(M_{n-1}-m_{n-1})-\tilde K\langle y_1(t)-y_2(t),v \rangle \\
& - \tilde \psi_0 (\langle y_1(t-\tilde\tau_1),v \rangle-\langle y_2(t-\tilde\tau_2),v \rangle ) \\
& \leq \left(\tilde K-\tilde \psi_0\right)(M_{n-1}-m_{n-1})-\tilde K\langle y_1(t)-y_2(t),v \rangle,
\end{split}
\end{equation*}
where, since $t-\tilde\tau_1,t-\tilde\tau_2 \geq (n-2)\tau,$ we use the initial assumption. Therefore, applying the  Grönwall inequality we find that, for $t=n\tau$,
$$ d(n\tau) \leq \left( 1-\frac{\tilde \psi_0}{\tilde K}(1-e^{-\tilde K\tau})\right) D_{n-2}.$$
Thus, the result follows for a suitable constant $C.$
\end{proof}
Again, from this result we can deduce the exponential consensus estimate for the case of two leaders.
\begin{Theorem}\label{maincont_2_leaders}
Every solution
$\left(\{x_i(t)\}_{i=1}^N,\{y_j(t)\}_{j=1}^2\right)$  to \eqref{eqDUE}-\eqref{IC2} satisfies the exponential decay estimate
\begin{equation*} 
d(t)\leq D(0)e^{-\bar{\gamma} (t-2\tau)} \ \mbox{for all} \ t \geq 0,
\end{equation*}
for a suitable positive constant $\bar{\gamma}.$
\end{Theorem}

\section{Numerical tests} \label{num}
In this section, we present some numerical simulations illustrating the theoretical results. \\
In Figure 1, we consider the case with one leader with  constant trajectory. We consider the weight functions $b_{ij}(t)$ defined by 
$$\psi(r,r'):=\tilde{\psi}(|r-r'|), \quad r,r'\in [0, +\infty),$$
for all $i=1,\dots,N.$
Meanwhile, the weight functions $c_{i0}(t)$  are assumed to be the same for all $i=1,\dots,N$.
In particular, we assume
\begin{equation} \label{simulations}
\begin{split}
& \tilde{\psi}(r):= e^{-(r-1)^2}, \ r \in [0,+\infty), \\
& c_{i0}(t):= \frac{\tilde{K}}{N}, \ \forall \ i \in \{1,\dots,N\},\\
\end{split}
\end{equation}
with $\tilde{K}$ is a positive constant. Moreover, for simplicity, we assume a constant time delay $\tau_{ij}=5,$ for all $i,j=1,\dots,N.$

\begin{figure} 
\centering
\includegraphics[scale=0.35]{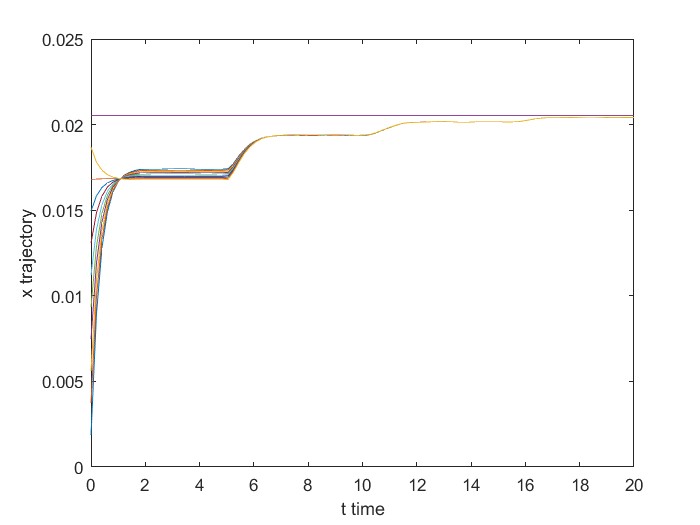}
\caption{HK-model with a unique leader with constant trajectory}
\end{figure}

In Figure 2, we present the case with one leader with a controlled trajectory. Here, we use the weight functions defined in \eqref{simulations} and a constant time delay $\tau_{ij}=1,$ for all $i=1,\dots,N, \ j=0,\dots,N.$
\begin{figure}
\centering
\includegraphics[scale=0.3]{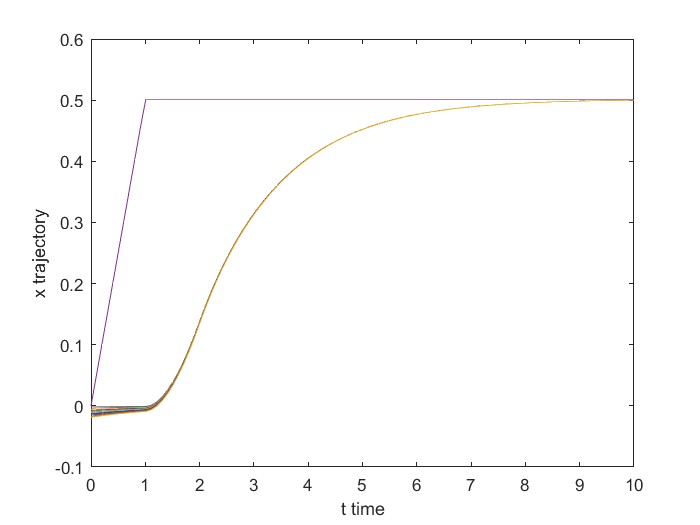}
\includegraphics[scale=0.3]{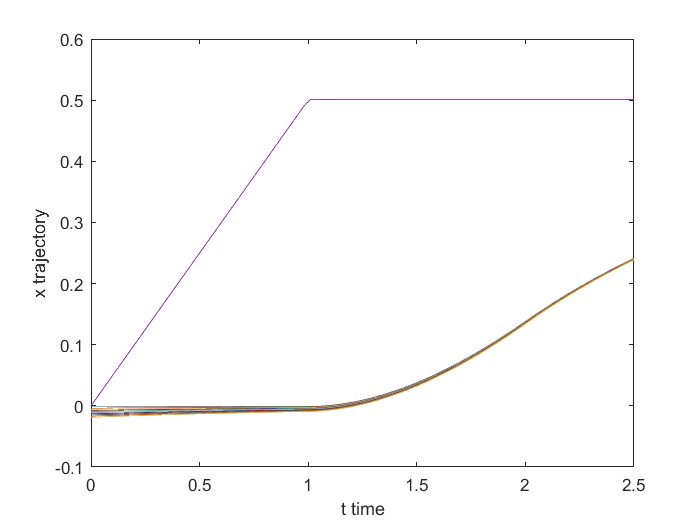}
\caption{HK-model with a unique leader with control}
\end{figure}

In Figure 3, we consider the case with two leaders. Again, we choose the functions $a_{ij}(t)$ for $i,j=1,2,$ and $b_{ij}(t)$ for $i,j=1,\dots,N,$ equal to the function $\tilde{\psi}(r)$ in \eqref{simulations}. Meanwhile, the function $c_{ij}(t)$ is supposed constant and equal to $\frac{\tilde{K}}{N+1}.$ The time delays $\tilde\tau_j$ for $j=1,2$ is constant and $\tilde\tau_j=5.$
\begin{figure}
\centering
\includegraphics[scale=0.35]{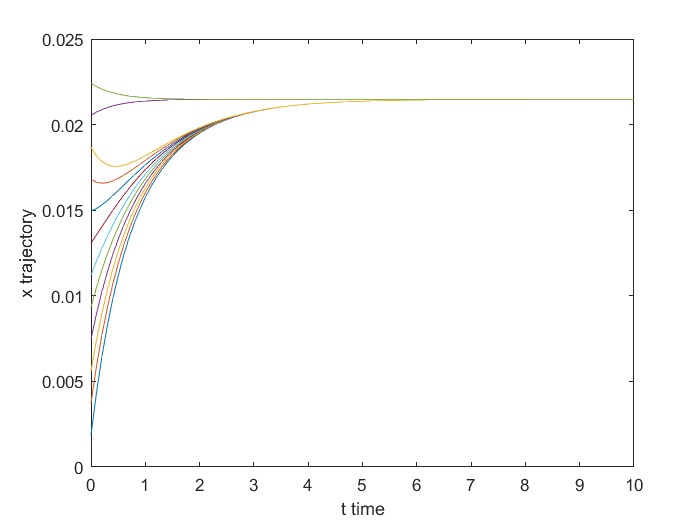}
\caption{HK-model with two leaders}
\end{figure}
For simplicity, we assumed $\chi_{ij}=1$ in all the figures above.
Finally, in Figure 4, we plot the general case in which the Common Influencer assumption \textbf{(CI)} holds. As above, the function $a_{ij}(t)=\tilde{\psi}(r)$ and the time delays $\tau_{ij}=5$ for all $i,j=1,\dots,N.$
\begin{figure} 
\centering
\includegraphics[scale=0.35]{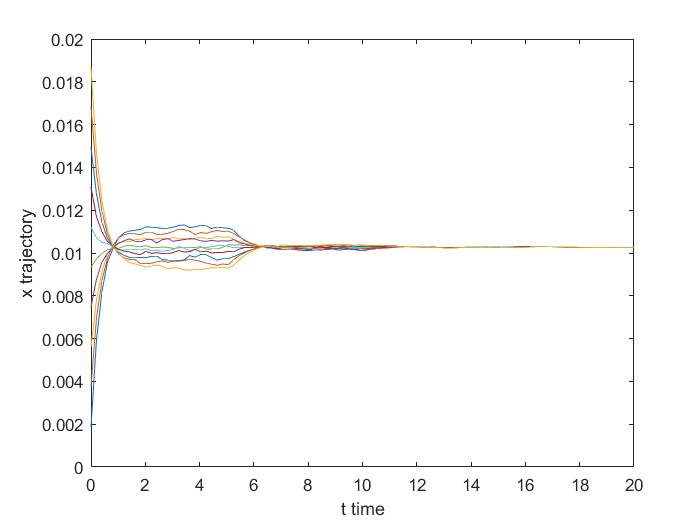}
\caption{HK-model under common influencer assumption}
\end{figure}

\bigskip
\eject
	
	\noindent {\bf Acknowledgements.} {\small C. Cicolani and C. Pignotti are members of Gruppo Nazionale per l’Analisi Ma\-te\-ma\-ti\-ca,
la Probabilità e le loro Applicazioni (GNAMPA) of the Istituto Nazionale di
Alta Matematica (INdAM).  C. Pi\-gnot\-ti is partially
supported by PRIN 2022  (2022238YY5) {\it Optimal control problems: analysis,
approximation and applications}, PRIN-PNRR 2022 (P20225SP98) {\it Some mathematical approaches to climate change and its impacts}, and by INdAM GNAMPA Project {\it Modelli alle derivate parziali per interazioni multiagente non 
simmetriche} (CUP E53C23001670001). }

\end{document}